\documentclass[10pt]{amsart}
\usepackage[a4paper,margin=1.00in]{geometry}
\usepackage{amsmath,amssymb,amsthm}
\usepackage{xcolor}
\usepackage{enumitem}
\usepackage[colorlinks=true,urlcolor=blue,citecolor=red,
linkcolor=blue,linktocpage,pdfpagelabels,
bookmarksnumbered,bookmarksopen]{hyperref}
\usepackage[hyperpageref]{backref}
\usepackage{etoolbox}
\makeatletter
\patchcmd{\@settitle}
  {\uppercasenonmath\@title}
  {}{}{}
\patchcmd{\@setauthors}
  {\MakeUppercase{\authors}}
  {\authors}{}{}
\makeatother

\allowdisplaybreaks[2]
\numberwithin{equation}{section}
\newtheorem{theorem}{Theorem}[section]
\newtheorem{proposition}[theorem]{Proposition}
\newtheorem{lemma}[theorem]{Lemma}
\newtheorem{corollary}[theorem]{Corollary}
\theoremstyle{definition}
\newcommand{\R}{\mathbb R}
\newcommand{\abs}[1]{\lvert#1\rvert}

\title[Odd cubic radial bound states]{Uniqueness of radial bound states for an equation with an odd cubic-like nonlinearity}
\author{Wenjing Chen and Zexi Wang\textsuperscript{\dag}}
\date{}
\address{School of Mathematics and Statistics, Southwest University,
Chongqing 400715, China}
\email{wjchen@swu.edu.cn}\email{zxwangmath@163.com}
\thanks{\textsuperscript{\dag}Corresponding author. E-mail address:
zxwangmath@163.com.}
\thanks{The research has been supported by National Natural Science Foundation of China (No. 12671143).}
\begin{document}
\begin{abstract}
We investigate radial bound states with prescribed nodes for the odd cubic equation
\[
 \Delta u+u(\abs{u}-b)(c-\abs{u})=0
 \quad\text{in }\R^n,\qquad n\ge2,\quad 0<b<c.
\]
Writing $a=b/c$, we prove that a radial bound state with exactly
$k\ge1$ simple zeros and positive central value exists if and only if
$0<a<1/2$. For each $k$, this solution is unique for all $0<a<1/2$
in dimension two, for $a\ge(7-3\sqrt5)/2$ in dimension three, and
for $a\ge(43-5\sqrt{33})/32$ in dimension four.
For one node, we identify exact ranges obtained from a tangent
comparison and a quotient inequality in every dimension $n\ge3$;
the threshold is dimension-independent for $3\le n\le23$ and is
characterized implicitly for $n\ge24$.
These results address the natural odd counterpart of the cubic
nonlinearity highlighted in Tang's prescribed-node uniqueness
conjecture [M. Tang, \emph{Invent. Math.}, 2026].
The main challenges are the second positive zero of the cubic source,
the lack of a global comparison inequality, and supercritical cubic
growth in higher dimensions. We overcome them by an amplitude
confinement argument, controlled extensions preserving sharp
comparison estimates, and a subcritical truncation for existence.
A complementary small-$a$ uniqueness range in dimensions three to
five follows from a recent massive double-power result.
\end{abstract}
\maketitle

\markboth{W. Chen and Z. Wang}{Odd cubic radial bound states}

\section{Introduction and main results}

We study the semilinear elliptic equation
\begin{equation}\label{eq:original}
 \Delta u+u(\abs{u}-b)(c-\abs{u})=0
 \quad\text{in }\R^n,\qquad n\ge2,\quad 0<b<c.
\end{equation}
A \emph{radial bound state} is a nontrivial classical radial solution
converging to zero as $\abs{x}\to\infty$.
We normalize the sign by requiring $u(0)>0$.
{A radial $k$-node bound state has exactly $k$ zeros in $(0,\infty)$.}
The zeros are simple, and the
solutions considered below have exponential decay. Throughout the main results, $k\ge1$.

The scalar-field problem $\Delta u+f(u)=0$ has a long history,
particularly in connection with standing waves of nonlinear
Schr\"odinger and Klein--Gordon equations.
The foundational work of Berestycki and Lions
\cite{BL1983I,BL1983II} gave general hypotheses for the existence of
positive and infinitely many radial nodal solutions; see also the earlier
existence work of Strauss \cite{Strauss1977}.
Prescribing an exact number of zeros is a more delicate task than constructing a sequence of critical points with undetermined type.
Jones and K\"upper \cite{JonesKupper1986} used dynamical systems,
McLeod, Troy and Weissler \cite{MTW1990} developed radial shooting
methods, and Grillakis \cite{Grillakis} established a variational and
topological existence theorem for prescribed-node radial solutions.
These developments provide the existence framework for
\eqref{eq:original}.

Uniqueness, by contrast, is more sensitive to the structure of
the nonlinearity, particularly for sign-changing solutions.
For positive states, fundamental results include
Kwong's uniqueness theorem for pure-power nonlinearities
\cite{Kwong1989}, the separation method of Serrin and Tang
\cite{SerrinTang2000}, and the work of Lewin and Rota Nodari
on massive double-power nonlinearities \cite{LRN}.
However, uniqueness of positive states does not imply uniqueness
of nodal states, and nonuniqueness may occur for more general
nonlinearities \cite{CGHH2015,DavilaDelPinoGuerra2013}.
For nodal states, Cort\'azar,
Garc\'ia-Huidobro and Yarur established uniqueness
criteria for solutions with exactly one node \cite{CGHY2009}
and subsequently developed comparison criteria ensuring
uniqueness for any prescribed positive number of nodes
in low dimensions \cite{CGHY}.
Related results for weighted equations were obtained
in \cite{CGHH}.
These comparison methods form the basis of our uniqueness
analysis.

Recently Tang \cite{Tang} settled the Berestycki--Lions
prescribed-node uniqueness conjecture for the massive pure-power model
\begin{equation*}
 \Delta Q-Q+\abs{Q}^{p-1}Q=0
 \quad\text{in }\R^n,\qquad
 n\ge3,\quad 1<p<\frac{n+2}{n-2}.
\end{equation*}
{A key ingredient in Tang's proof is the analysis of the radial shooting variation and the transitions in its nodal structure.}
Earlier advances in this direction include the nearly critical analysis
of Ao, Wei and Yao \cite{AoWeiYao2016}, as well as the computer-assisted
uniqueness result of Cohen, Li and Schlag \cite{CohenLiSchlag2024} for the three-dimensional cubic equation \(\Delta u-u+u^3=0\).
Beyond the pure-power case, Tang also raised analogous prescribed-node uniqueness questions for several other nonlinearities, including a cubic-type model. This provides a principal motivation for the present work, in which we investigate the corresponding odd cubic-type nonlinearity.

\medskip
\noindent\textbf{Tang's conjecture (cubic model; \cite{Tang}).}
For the cubic nonlinearity
\begin{equation}\label{eq:Tang-cubic}
 f(s)=s(s-b)(c-s),\qquad 0<b<c,
\end{equation}
$\Delta u+f(u)=0$ admits at most one radial $k$-node bound
state with its absolute maximum at the origin of $\mathbb{R}^n$, for any given $k\geq1$ and $n\geq2$.
\medskip

The behavior of the nonlinearity on the negative half-line is essential when considering nodal bound states. Indeed, the polynomial in \eqref{eq:Tang-cubic} satisfies $f(s)>0$ for all $s<0$. Consequently, any solution of $\Delta u+f(u)=0$ that tends to zero at infinity cannot take negative values. Otherwise, it would attain a negative minimum at some point $x_0$, where $\Delta u(x_0)\ge0$, contradicting $\Delta u(x_0)=-f(u(x_0))<0$. Thus, the original polynomial model admits no nodal bound states. To formulate a meaningful nodal problem while preserving the nonlinearity on the positive half-line, we consider its odd extension
\[
f_{b,c}(s)=s(|s|-b)(c-|s|),
\qquad f_{b,c}(-s)=-f_{b,c}(s),
\]
which coincides with the nonlinearity in \eqref{eq:original}. Our results therefore concern the odd cubic-type counterpart of Tang's proposed model, rather than the original nonodd polynomial.

Unlike the pure-power model, the cubic nonlinearity in
\eqref{eq:original} has a second positive zero and exhibits Sobolev-critical or supercritical growth in dimensions $n\ge4$. These features prevent a direct application of standard prescribed-node existence and uniqueness results. {We first confine the amplitudes of all bound states to $(-1,1)$ and then use two different modifications of the nonlinearity: a subcritical, linearly growing extension for existence and a positivity- and comparison-preserving extension for uniqueness.}

Set
\[
 a=\frac bc\in(0,1),\qquad u(x)=c\,v(cx),\qquad
 g_a(s)=s(\abs{s}-a)(1-\abs{s}).
\]
The normalized radial shooting problem is
\begin{equation}\label{eq:radial}
 v''+\frac{n-1}{r}v'+g_a(v)=0,\qquad
 v(0)=\alpha>0,\quad v'(0)=0,\quad
 \lim_{r\to\infty}v(r)=0.
\end{equation}
Its potential is
\[
 G_a(s)=\int_0^s g_a(t)\,dt
 =-\frac{\abs{s}^4}{4}+\frac{1+a}{3}\abs{s}^3-\frac a2s^2.
\]
A bound state is strictly confined to $\abs{v}<1$. The radial
energy decreases to zero and is strictly positive at each finite
radius. Because $G_a(1)=(1-2a)/12$, this immediately rules out
bound states for $1/2\le a<1$.
For $0<a<1/2$, let
\[
 \beta_a=\frac{2(1+a)-\sqrt{4(1+a)^2-18a}}3.
\]
Then $a<\beta_a<1$, $G_a(\beta_a)=0$, and
$G_a>0$ on $(\beta_a,1]$. The central height of every
bound state lies in $(\beta_a,1)$, and the absolute maximum is
attained only at the origin.
Let $\mathcal B_{n,k}(a)$ denote the set of $k$-node solutions of
\eqref{eq:radial}.

We first establish a sharp existence criterion for radial bound states with any prescribed positive number of nodes in every dimension $n\ge2$.

\begin{proposition}\label{prop:existence}
Let $n\ge2$, $k\ge1$ and $0<a<1$.
Then
\[
\mathcal B_{n,k}(a)\ne\varnothing
\quad\Longleftrightarrow\quad 0<a<\frac12.
\]
Every state decays exponentially, and its absolute maximum is
attained only at the origin.
\end{proposition}

To establish uniqueness, we examine the comparison criteria of \cite{CGHY} on $(\beta_a,1)$, which contains the central amplitudes of all bound states. Unlike the pure-power case, the identities
\[
sg_a'(s)-g_a(s)=s^2(1+a-2s),\qquad
sg_a(s)-2G_a(s)=\frac{s^3}{6}\bigl(2(1+a)-3s\bigr),\qquad s>0
\]
show that both quantities change sign on $(a,1)$ when $0<a<1/2$. Thus, the sign arguments employed in \cite[Lemma~4.2]{Tang} cannot be applied directly. Instead, we derive an explicit formula for $(G_a/g_a)'$ and identify the sharp parameter ranges for the relevant comparison inequalities. For any two hypothetical radial bound states, we then construct a common odd $C^1$ extension that agrees with $g_a$ on their amplitude ranges and satisfies the global hypotheses of \cite{CGHY}. This enables us to apply the corresponding uniqueness theorems to the original problem.

Our first uniqueness result holds for {every} prescribed positive
node number.
\begin{theorem}\label{thm:explicit}
Let $k\ge1$ and $0<a<1/2$.
Then $\mathcal B_{n,k}(a)$ consists of exactly one state in each
of the following cases:
\begin{align*}
&n=2;\\
&n=3,\qquad a\ge a_3:=\frac{7-3\sqrt5}{2};\\
&n=4,\qquad a\ge a_4:=\frac{43-5\sqrt{33}}{32}.
\end{align*}
\end{theorem}
For radial bound states with exactly one node, we establish uniqueness in every dimension $n\ge3$ under weaker comparison conditions.
Define $a_T\in(0,a_3)$ to be the
unique number satisfying
\[
 \beta_{a_T}=\frac{(1+a_T)^3}{9(1-a_T+a_T^2)},
\]
and put
\begin{equation}\label{defm}
 m(a)=\inf_{\beta_a<s<1}\left(\frac{G_a}{g_a}\right)'(s).
\end{equation}
Proposition~\ref{prop:tangent} proves that $a_T$ is well defined,
and Proposition~\ref{prop:minimum} computes $m(a)$ explicitly.
\begin{theorem}\label{thm:one}
Let $n\ge3$ and $0<a<1/2$.
The set $\mathcal B_{n,1}(a)$ consists of exactly one state whenever
\[
\begin{cases}
a\ge a_T,&3\le n\le23,\\
a\ge\gamma_n,&n\ge24,
\end{cases}
\]
where, for each $n\ge24$, $\gamma_n$ is the unique number in
$(a_T,a_3)$ satisfying
\[
 m(\gamma_n)=\frac{n-2}{2n}.
\]
The sequence $(\gamma_n)_{n\ge24}$ is strictly increasing and
converges to $a_3$.
For each $n\ge3$, the indicated interval is exactly the range
in which the two comparison conditions
\begin{equation}\label{eq:one-criterion}
a\ge a_T,\qquad m(a)\ge\frac{n-2}{2n}
\end{equation}
hold simultaneously.
\end{theorem}

The two theorems reveal a dimension-dependent picture of uniqueness. In dimension two, uniqueness holds throughout the entire existence range for every prescribed positive number of nodes. In dimensions three and four, explicit algebraic thresholds guarantee the same conclusion. For one-node solutions, a tangency condition combined with a weaker quotient inequality extends uniqueness to all dimensions $n\ge3$, with a uniform threshold for $3\le n\le23$ and dimension-dependent thresholds for $n\ge24$. These thresholds are sharp for the comparison criteria employed here, but are not claimed to be necessary for uniqueness itself.

A complementary uniqueness result for sufficiently small $a$ in dimensions $3\le n\le5$ follows from the recent work of Guo, Ji and Luo \cite{GJL} on massive double-power equations. A simple rescaling yields the following corollary, independently of the proofs of the preceding two theorems.

\begin{corollary}\label{cor:small}
Fix $n\in\{3,4,5\}$ and $k\ge1$.
There exists $a_{n,k}\in(0,1/2)$ such that
$\mathcal B_{n,k}(a)$ consists of exactly one state for every
$0<a<a_{n,k}$.
\end{corollary}
The reduction to \cite{GJL} is presented at the end of Section~\ref{sec:criteria}. However, no quantitative relation between $a_{n,k}$ and the explicit comparison thresholds is established.
Thus, while existence holds throughout $0<a<1/2$, uniqueness is established in the parameter ranges specified above.

The remainder of the paper is organized as follows. Section~\ref{sec:existence} establishes amplitude confinement, energy and decay estimates, and prescribed-node existence. Section~\ref{sec:criteria} derives sharp comparison bounds and proves the uniqueness results through a common extension of the nonlinearity.

\section{Existence and amplitude bounds}\label{sec:existence}

Write $d=n-1$. At the origin, the radial equation is understood
through
\begin{gather*}
v'(r)=-r^{-d}\int_0^r s^d g_a(v(s))\,ds,\\
v(r)=\alpha-\int_0^r t^{-d}\int_0^t s^d g_a(v(s))\,ds\,dt.
\end{gather*}
The Lipschitz constant of the local integral map is bounded by \(Cr^2\), which yields regular local existence and uniqueness.
In particular, $v''(0)=-g_a(\alpha)/n$.
The nonlinearity is $C^1$ with locally Lipschitz first derivative.

\begin{lemma}\label{lem:decay}
Let $0<a<1$. Then every nontrivial classical solution of \eqref{eq:radial} satisfies $|v|<1$.
For every $0<\mu<\sqrt a$, there is a constant $C_\mu$ such that
\begin{equation}\label{eq:decay}
|v(r)|+|v'(r)|+|v''(r)|\leq C_\mu e^{-\mu r},
\qquad  r\geq1.
\end{equation}
Its zeros are simple and finite in number.
\end{lemma}

\begin{proof}
The positive maximum and the negative minimum are attained if the corresponding sign occurs.
At a positive maximum larger than one, $g_a(v)<0$.
The equation would give $\Delta v>0$, a contradiction.
If $v(r_0)=1$, then $v'(r_0)=0$.
Cauchy uniqueness when $r_0>0$, or regular local uniqueness
when $r_0=0$, gives $v\equiv1$, contradicting decay.
The argument at $-1$ is identical.

Choose $\eta>0$ so small that
\[
q(s):=-\frac{g_a(s)}{s}>\mu^2,\qquad 0<|s|<\eta.
\]
For sufficiently large $R$, $|v(r)|<\eta$ on $[R,\infty)$.
There is no zero in this interval.
Indeed, after a zero, a nonzero lobe which tends to zero or ends at another zero has an interior extremum.
At a positive such maximum the equation gives $v''=-g_a(v)>0$.
At a negative minimum, it gives the opposite contradiction.
A double zero would make the entire solution zero by Cauchy uniqueness.
After reflection, we may therefore assume $v>0$ on $[R,\infty)$.

The function $w(r)=v(R)e^{-\mu(r-R)}$ satisfies
\[
w''+\frac d r w'-q(v)w
=\left(\mu^2-\frac{d\mu}{r}-q(v)\right)w<0.
\]
The one-dimensional maximum principle on the exterior interval gives $0<v\leq w$.
It applies directly because $v-w\to0$ as $r\rightarrow\infty$ and a positive maximum would be attained at a finite radius.

The derivative estimates follow from the equation on intervals of fixed length.
For $r\geq R+2$, the mean value theorem gives a point $s\in[r-1,r+1]$ with
\[
|v'(s)|\leq\sup_{[r-1,r+1]}|v|.
\]
On $[r-1,r+1]$, the derivative satisfies the first-order equation
\[
 (v')'(t)+\frac{d}{t}v'(t)=-g_a(v(t)).
\]
Since $d/t$ is bounded on this interval and
$|g_a(v)|\le C|v|$ for $|v|\le1$, integration from $s$ to $r$
with the corresponding integrating factor yields
\[
 |v'(r)|\le C\Bigl(|v'(s)|+\sup_{[r-1,r+1]}|v|\Bigr)
 \le C\sup_{[r-1,r+1]}|v|.
\]
The radial equation then bounds $v''$.
This proves \eqref{eq:decay}.
Every finite zero is simple by Cauchy uniqueness.
{Such zeros cannot accumulate on a compact interval. Otherwise a sequence of zeros would accumulate at some $r_0>0$, forcing $v(r_0)=v'(r_0)=0$ and hence $v\equiv0$. Since $v(0)>0$ and there are no zeros beyond $R$, the total number of zeros is finite.}
\end{proof}

\begin{lemma}\label{lem:energy}
For a radial bound state with $0<a<1$, the energy satisfies
\begin{gather*}
E(r):=\frac12v'(r)^2+G_a(v(r)),\qquad
E'(r)=-\frac d r v'(r)^2,\qquad r>0,\\
E(r)=d\int_r^\infty\frac{v'(s)^2}{s}\,ds>0,\qquad r\ge0.
\end{gather*}
Consequently, $0<a<1/2$ and
\[
\beta_a<v(0)<1,\qquad |v(r)|<v(0),\qquad r>0.
\]
\end{lemma}

\begin{proof}
Multiplying the radial equation by \(v'\) yields the differential identity,
and Lemma~\ref{lem:decay} gives
$E(\infty)=0$. The integral is positive: otherwise the solution
would be zero on a terminal interval and hence identically zero
by uniqueness.

On $[0,1]$, the maximum of $G_a$ is
$\max\{0,G_a(1)\}$, with $G_a(1)=(1-2a)/12$.
Thus $a\ge1/2$ contradicts
$G_a(v(0))=E(0)>0$ and $v(0)<1$.
For $0<a<1/2$, the positive roots of $G_a$ are $\beta_a\in(a,1)$
and a root greater than one; hence $v(0)>\beta_a$.
Since $g_a(v(0))>0$ and $v'(0)=0$, the profile initially decreases.
At a first return to $|v|=v(0)$ at a positive radius, its energy
would be at least $G_a(v(0))=E(0)$.
Strict energy loss excludes this return, proving the amplitude bound.
\end{proof}

\begin{proof}[Proof of Proposition~\ref{prop:existence}]
Nonexistence for $1/2\le a<1$ follows from Lemma~\ref{lem:energy}.
Fix $0<a<1/2$. To apply \cite[Theorem~1.1]{Grillakis}, define
\[
\widehat g_a(s)=
\begin{cases}
g_a(s),&|s|\le1,\\
-(1-a)\bigl(s-\operatorname{sgn}s\bigr),&|s|>1.
\end{cases}
\]
The values and first derivatives agree at $\pm1$, because
$g_a(1)=0$ and $g_a'(1)=a-1$.
Thus $\widehat g_a$ is odd, $C^1$ and globally Lipschitz.

In the sign convention of \cite{Grillakis}, take
$f=-\widehat g_a$ in $-\Delta v+f(v)=0$.
Its positive mass is $f'(0)=a$.
Near zero,
\[
as+\widehat g_a(s)=(1+a)|s|s-s^3=O(s^2),
\]
whereas this expression grows at most linearly at infinity.
Consequently,
\[
|as+\widehat g_a(s)|\le C_{a,\sigma}|s|^{1+\sigma},
\qquad s\in\mathbb R,\quad 0<\sigma<1.
\]
For $n\ge3$, choose $\sigma$ additionally below $4/(n-2)$.
Choose $0<\sigma_1<\sigma_2<\min\{1,4/(n-2)\}$ for
$n\ge3$, or any $0<\sigma_1<\sigma_2<1$ for $n=2$.
Writing $as+\widehat g_a(s)=g_1(s)+g_2(s)$ with
$g_1=g_2=(as+\widehat g_a)/2$, the preceding estimates give
$|g_j(s)|\le C_j|s|^{1+\sigma_j}$ globally for $j=1,2$.
This is precisely the subcritical decomposition required in
(A2) of \cite{Grillakis}; (A1) holds with positive mass $a$.

Put $\widehat G_a(s)=\int_0^s\widehat g_a(t)\,dt$.
The first nonzero roots of the source potential $-\widehat G_a$
are $\pm\beta_a$, and
\[
-\widehat G_a(\pm\beta_a)=0,\qquad
-\widehat g_a(\pm\beta_a)\ne0,\qquad
-\widehat G_a(\pm1)=-\frac{1-2a}{12}<0.
\]
Thus the potential $F=-\widehat G_a$ satisfies (A3) of \cite{Grillakis}.
 Moreover,
\[
 -\widehat G_a(s)
=-G_a(1)+\frac{1-a}{2}(|s|-1)^2,\qquad |s|\ge1,
\]
so $-\widehat G_a(s)\to\infty$ as $|s|\to\infty$.
Grillakis's theorem therefore gives a radial
$v\in H^1(\mathbb R^n)$ with exactly $k$ nodes satisfying
$-\Delta v-\widehat g_a(v)=0$.

It remains to show that this solution solves the original equation.
Testing with $(v-1)_+\in H^1(\mathbb R^n)$ yields
\[
\int_{\mathbb R^n}|\nabla(v-1)_+|^2
=-(1-a)\int_{\mathbb R^n}(v-1)_+^2.
\]
Hence $v\le1$. Applying the same argument to $-v$ gives $v\ge-1$.
Thus $\widehat g_a(v)=g_a(v)$ everywhere.

Since $|v|\le1$ and $v\in L^2$, both $v$ and $g_a(v)$ belong
to $L^p$ for every finite $p\ge2$.
{Applying the global $L^p$ estimate to $(-\Delta+1)v=v+g_a(v)$ gives $v\in W^{2,p}(\mathbb R^n)$ for 
every finite $p\ge2$.}
Choosing $p>n$ and then applying local elliptic regularity
gives a classical solution. As
$v\in H^1_{\mathrm{rad}}(\mathbb R^n)$ and $n\ge2$,
the radial lemma \cite[Lemma A.II]{BL1983I} yields $v(r)\to0$ as $r\to\infty$. By uniqueness for the regular radial initial-value
problem, $v(0)\ne0$. Replacing $v$ by $-v$ if necessary,
we may assume $v(0)>0$. Finally,
Lemmas~\ref{lem:decay} and~\ref{lem:energy} yield
the strict amplitude bound, exponential decay,
and the absolute-maximum assertion.
This completes the proof.
\end{proof}

\section{Uniqueness by comparison}\label{sec:criteria}

We now apply the comparison criteria of \cite{CGHY}.
Let $f$ be odd and $C^1$, with $f<0$ on $(0,b_0)$ and $f>0$ on $(b_0,\infty)$.
Write $F(s)=\int_0^s f(t)\,dt$.
Suppose $F(\beta)=0$ for some $\beta>b_0$ and $F(s)\to\infty$ as $s\to\infty$.
By \cite[Theorem~1.1]{CGHY}, for $2\leq n\leq4$, the condition
\begin{equation}\label{eq:CGH-all}
\left(\frac Ff\right)'(s)\geq\frac{n-2}{2},\qquad s>\beta
\end{equation}
implies $\Delta u+f(u)=0$ admits at most one radial bound state with any prescribed positive number of nodes.
For $n\geq2$, \cite[Theorem~1.2]{CGHY} gives at most one 1-node radial bound state under the conditions
\begin{equation}\label{eq:CGH-one}
f(s)\geq(s-\beta)f'(s),\qquad
\left(\frac Ff\right)'(s)\geq\frac{n-2}{2n},
\qquad s>\beta.
\end{equation}
{In both statements, only solutions with positive central value are compared. The first inequality in \eqref{eq:CGH-one} also holds at $s=\beta$ by continuity.} The $C^1$ assumption also ensures $f'\in L^1(0,1)$, as required in \cite{CGHY}.

Since $g_a(s)$ becomes negative for $s>1$, it does not satisfy the global sign assumption of the comparison theorem. However, both states under comparison have amplitudes strictly below $1$. We may therefore choose $M<1$ larger than their maximum amplitudes and modify $g_a$ only for $|s|>M$ to meet this assumption, without affecting either state.

\begin{lemma}\label{lem:extension}
Let $0<a<1/2$ and $v_1,v_2$ be radial bound states for $g_a$.
Suppose, on $(\beta_a,1)$, that
\[
\left(\frac{G_a}{g_a}\right)'\geq C_0,\qquad C_0\leq1.
\]
There exists an odd $C^1$ function $\widetilde g$ on $\mathbb R$ for which both $v_1,v_2$ remain solutions and, with $\widetilde G(s)=\int_0^s\widetilde g(t)\,dt$,
\begin{gather*}
\widetilde g<0\text{ on }(0,a),\qquad
\widetilde g>0\text{ on }(a,\infty),\\
\widetilde G(\beta_a)=0,\qquad
\widetilde G(s)\rightarrow\infty\text{ as }s\to\infty,\qquad
\left(\frac{\widetilde G}{\widetilde g}\right)'\geq C_0\text{ on }(\beta_a,\infty).
\end{gather*}
If $g_a(s)\geq(s-\beta_a)g_a'(s)$ on $(\beta_a,1)$, the extension also satisfies this inequality on $(\beta_a,\infty)$.
\end{lemma}

\begin{proof}
By Lemma~\ref{lem:energy}, both positive central amplitudes
lie strictly below $1$ and dominate the absolute values of the
corresponding profiles. Since $g_a'(1)=a-1<0$, we can
choose $M\in(\beta_a,1)$ larger than both central amplitudes
and sufficiently close to $1$ that $g_a'(M)<0$.
The two profiles are then confined to $(-M,M)$.
Choose $h>0$ with $h\le-g_a(M)/(2g_a'(M))$, and define
\[
\widetilde g(s)=
\begin{cases}
g_a(s),&0\le s\le M,\\
g_a(M)+h g_a'(M)\bigl(1-e^{-(s-M)/h}\bigr),&s> M.
\end{cases}
\]
Extend this function oddly. The values and first derivatives
agree at $M$, so $\widetilde g\in C^1(\mathbb R)$.
For $s\geq M$,
\[
\widetilde g(s)\ge\tfrac12g_a(M)>0,\qquad
\widetilde g'(s)=g_a'(M)e^{-(s-M)/h}<0.
\]
Thus $\widetilde G(s)\to\infty$ and
$\widetilde G(s)\geq G_a(M)>0$ on this interval.
It follows that
\[
\left(\frac{\widetilde G}{\widetilde g}\right)'
=1-\frac{\widetilde G\widetilde g'}{\widetilde g^2}
>1\ge C_0,\qquad
\widetilde g(s)>(s-\beta_a)\widetilde g'(s).
\]
Below $M$, all required inequalities are those of the original
nonlinearity.
\end{proof}

\begin{proposition}\label{prop:minimum}
Let $0<a<1/2$ and $a<s<1$.
Then
\begin{equation}\label{eq:quotient}
\left(\frac{G_a}{g_a}\right)'(s)
=\frac14+
\frac{a^2(2-a)}{12(1-a)(s-a)^2}
+\frac{1-2a}{12(1-a)(1-s)^2}.
\end{equation}
Put
\[
A=\frac{a^2(2-a)}{12(1-a)^3},\qquad
B=\frac{1-2a}{12(1-a)^3},\qquad
s_*=a+(1-a)\frac{A^{\frac{1}{3}}}{A^{\frac{1}{3}}+B^{\frac{1}{3}}}.
\]
The quantity in \eqref{defm} is
\begin{equation}\label{eq:minimum}
m(a)=\begin{cases}
\displaystyle\frac14+\left(A^{\frac{1}{3}}+B^{\frac{1}{3}}\right)^3,&s_*\geq\beta_a,\\[6pt]
1,&s_*<\beta_a.
\end{cases}
\end{equation}
\end{proposition}

\begin{proof}
A straightforward partial-fraction decomposition yields
\[
\frac{G_a(s)}{g_a(s)}
=\frac s4-\frac{1+a}{12}
-\frac{a^2(2-a)}{12(1-a)(s-a)}
+\frac{1-2a}{12(1-a)(1-s)}.
\]
Differentiation proves \eqref{eq:quotient}.
Write the right-hand side of \eqref{eq:quotient} as $D(s)$.
Then
\[
D'(s)=2(1-a)^2\left(-\frac{A}{(s-a)^3}+\frac{B}{(1-s)^3}\right),
\qquad
D''(s)=6(1-a)^2\left(\frac{A}{(s-a)^4}+\frac{B}{(1-s)^4}\right)>0.
\]
Thus the unique minimum on $(a,1)$ is attained at $s_*$.
Substitution gives the first expression in \eqref{eq:minimum}.
If $s_*<\beta_a$, $D(s)$ increases on $(\beta_a,1)$.
Since $G_a(\beta_a)=0$ and $g_a(\beta_a)>0$, its infimum there is $D(\beta_a)=1$.
\end{proof}

\begin{proposition}\label{prop:thresholds}
{For every $0<a<1/2$, the inequality $(G_a/g_a)'>0$ holds on $(\beta_a,1)$.}
Moreover,
\begin{align}
(G_a/g_a)'\geq\tfrac12\text{ on }(\beta_a,1)
&\quad\Longleftrightarrow\quad a\geq a_3,
\label{eq:threshold-three}\\
(G_a/g_a)'\geq1\text{ on }(\beta_a,1)
&\quad\Longleftrightarrow\quad a\geq a_4.
\label{eq:threshold-four}
\end{align}
Since $(G_a/g_a)'(\beta_a)=1$, the inequality $(G_a/g_a)'\ge(n-2)/2$ cannot hold throughout $(\beta_a,1)$ when $n\ge5$. Furthermore, $m(a)$ is continuous and strictly increasing on $(0,a_3]$ and satisfies $\lim_{a\downarrow0}m(a)=1/3$, $m(a_3)=1/2$.
\end{proposition}

\begin{proof}
The first assertion follows immediately from \eqref{eq:quotient}.

We begin with \eqref{eq:threshold-three}. Set $L=1+a$ and $D=\sqrt{2L^2-9a}$. Then $D>0$ for every $0<a<1/2$.
A direct expansion yields
\[
g_a(s)^2-2G_a(s)g_a'(s)=\frac{s^3}{6}P_a(s),
\]
where
\[
P_a(s):=-3s^3+6Ls^2-(2L^2+9a)s+4La.
\]
Since
\[
P_a'(s)=-9(s-s_-)(s-s_+),\qquad
s_-=\frac{2L-D}{3},\quad s_+=\frac{2L+D}{3},
\]
and
\[
\beta_a=\frac{2L-\sqrt2D}{3}<s_-,
\qquad
D^2-(1-2a)^2=(1-2a)(1+a)>0,
\]
we have $\beta_a<s_-<1<s_+$. Thus $P_a$ attains
its minimum on $[\beta_a,1]$ at $s_-$.
Putting $x={a}/{L^2}$, $\delta={D}/{L}=\sqrt{2-9x}$,
and substituting $s_-=L(2-\delta)/3$, we obtain
\[
\frac{P_a(s_-)}{L^3}
=\frac{2\delta^2(1-\delta)}{9}.
\]
Consequently, $P_a\ge0$ on $(\beta_a,1)$ if and only if
$\delta\le1$, equivalently $(1+a)^2\le9a$. Since
$0<a<1/2$, the latter condition is equivalent to
\[
a\ge\frac{7-3\sqrt5}{2}=a_3.
\]
The identity
\[
\left(\frac{G_a}{g_a}\right)'-\frac12
=\frac{g_a^2-2G_ag_a'}{2g_a^2}
\]
now proves \eqref{eq:threshold-three}.

Next, $G_a,g_a>0$ on $(\beta_a,1)$, and hence
\[
\left(\frac{G_a}{g_a}\right)'\ge1
\quad\Longleftrightarrow\quad g_a'\le0.
\]
The larger root of $g_a'(s)=-3s^2+2Ls-a$ is $s_m=\left(L+\sqrt{1-a+a^2}\right)/{3}$,
while the smaller root is below $a$. Therefore the
inequality holds throughout $(\beta_a,1)$ precisely when
$\beta_a\ge s_m$. In terms of
$x=a/L^2\in(0,2/9)$, this condition becomes
\[
\sqrt{4-18x}+\sqrt{1-3x}\le1.
\]
The left-hand side is strictly decreasing on $(0,2/9)$
and equals $1$ at $x=16/75$. It follows that
\[
\beta_a\ge s_m
\quad\Longleftrightarrow\quad
\frac{a}{(1+a)^2}\ge\frac{16}{75}
\quad\Longleftrightarrow\quad
16a^2-43a+16\le0.
\]
For $0<a<1/2$, this is equivalent to
\[
a\ge\frac{43-5\sqrt{33}}{32}=a_4,
\]
which proves \eqref{eq:threshold-four}.

Finally, let $0<a\le a_3$. \eqref{eq:threshold-three} gives
$m(a)<1/2$ for $a<a_3$ and $m(a_3)=1/2$.
Since $(G_a/g_a)'(\beta_a)=1$, the minimizer
$s_*$ from Proposition~\ref{prop:minimum} must lie in
$(\beta_a,1)$. Consequently,
\eqref{eq:minimum}  gives
\[
m(a)=\frac14+\left(A^{\frac{1}{3}}+B^{\frac{1}{3}}\right)^3.
\]
Moreover,
\[
A'(a)=\frac{a(4-a)}{12(1-a)^4}>0,
\qquad
B'(a)=\frac{1-4a}{12(1-a)^4}>0
\]
since $0<a\le a_3<1/4$.
Hence $m$ is strictly increasing on $(0,a_3]$. 
{Continuity of $m$ follows from \eqref{eq:minimum}, since the two branches agree when $s_*=\beta_a$.} As $a\downarrow0$, we have $A\to0$ and $B\to1/12$,
so $\lim_{a\downarrow0}m(a)=1/3$.
\end{proof}

\begin{proposition}\label{prop:tangent}
Let $0<a<1/2$.
Then
\begin{equation}\label{eq:tangent-test}
g_a(s)\geq(s-\beta_a)g_a'(s)\text{ on }(\beta_a,1)
\quad\Longleftrightarrow\quad
\beta_a\geq\frac{(1+a)^3}{9(1-a+a^2)}.
\end{equation}
There exists a unique $a_T\in(0,a_3)$ satisfying
\[
\beta_{a_T}
=\frac{(1+a_T)^3}{9(1-a_T+a_T^2)}.
\]
Moreover, the comparison condition in \eqref{eq:tangent-test}
holds if and only if $a\ge a_T$. Furthermore,
\begin{equation}\label{eq:tangent-dimension}
\frac{21}{46}<m(a_T)<\frac{11}{24}.
\end{equation}
\end{proposition}
\begin{proof}
Set $L=1+a$, $H=1-a+a^2$, and
\[
T(s):=g_a(s)-(s-\beta_a)g_a'(s).
\]
Since $T'(s)=-(s-\beta_a)(2L-6s)$,
if $\beta_a<L/3$, the minimum of $T$ on $(\beta_a,1)$
is attained at $L/3$, where
\[
T\left(\frac{L}{3}\right)=\frac{H\beta_a}{3}-\frac{L^3}{27}.
\]
If $\beta_a\ge L/3$, then $T$ is increasing and
$T(\beta_a)=g_a(\beta_a)>0$. Moreover, $3H-L^2=(1-2a)(2-a)>0$,
so $L^3/(9H)<L/3$. These two cases yield
\[
T(s)\ge0\text{ on }(\beta_a,1)
\quad\Longleftrightarrow\quad
\beta_a\ge\frac{L^3}{9H},
\]
proving \eqref{eq:tangent-test}.

To determine $a_T$, put $x=a/L^2$ and $h=\beta_a/L$.
Then it follows from $G_a(\beta_a)=0$ that
\[
x=\frac{4h-3h^2}{6},\qquad 0<h<\frac23.
\]
The threshold equality is equivalent to
\[
Q(h):=27h^3-36h^2+18h-2=0.
\]
Since $Q'(h)=9(9h^2-8h+2)>0$,  $Q(0)=-2$, and $Q\left({2}/{3}\right)=2$,
there is a unique root in $(0,2/3)$. Both $x$ and $h$
increase strictly with $a$. At $a=a_3$, we have
\[
x=\frac19,\qquad h=\frac{2-\sqrt2}{3},
\qquad Q(h)=6-4\sqrt2>0.
\]
Hence $a_T\in(0,a_3)$, and the condition holds precisely
when $a\ge a_T$.

It remains to prove \eqref{eq:tangent-dimension}. At
$a=a_T$, the equality $h=1/[9(1-3x)]$ gives
\[
P(x_T)=0,\qquad x_T=\frac{a_T}{(1+a_T)^2},
\]
where
\[
P(x):=1458x^3-972x^2+198x-11.
\]
Since $P'(x)>0$ on $(0,1/9)$ and $P(7/78)<0<P(9/100)$,
we obtain $7/{78}<x_T<9/{100}$.

For $0<a\le a_3$, write $S=A+B$, where $A,B$ are
defined in Proposition~\ref{prop:minimum}. In terms of
$x=a/(1+a)^2$,
\[
S=\frac{1-3x}{12(1-4x)},\qquad
AB=\frac{x^2(2-9x)}{144(1-4x)^3},
\]
and
\[
\left(m(a)-\frac14-S\right)^3
=27AB\left(m(a)-\frac14\right).
\]
Since $t\mapsto(t-S)^3/t$ is strictly increasing for
$t>S$, direct substitution gives
\[
m(a)>\frac{n-2}{2n}
\quad\Longleftrightarrow\quad R_n(x)<0,
\]
provided $t_n:=(n-4)/(4n)>S$, where
\[
R_n(x)=
[2(n-6)-3(3n-16)x]^3
-81n^2(n-4)x^2(2-9x).
\]
At the two comparison points,
\[
\begin{aligned}
S\left(\frac{7}{78}\right)&=\frac{19}{200}<\frac{19}{92}=t_{23},\qquad
&&R_{23}\left(\frac{7}{78}\right)<0,\\
S\left(\frac{9}{100}\right)&=\frac{73}{768}<\frac5{24}=t_{24},\qquad
&&R_{24}\left(\frac{9}{100}\right)
>0.
\end{aligned}
\]
Thus the corresponding values of $m$ satisfy $m(a_1)>{21}/{46}$ and $m(a_2)<{11}/{24}$,
where $a_1,a_2$ are determined by
\[
\frac{a_1}{(1+a_1)^2}=\frac7{78},\qquad
\frac{a_2}{(1+a_2)^2}=\frac9{100}.
\]
Since $m$ is strictly increasing on $(0,a_3]$ and
$a_1<a_T<a_2$, we conclude \eqref{eq:tangent-dimension}.
\end{proof}

\begin{proof}[Proof of Theorem~\ref{thm:explicit}]
Existence follows from Proposition~\ref{prop:existence}.
To prove uniqueness, suppose that two distinct states
$v_1,v_2\in\mathcal B_{n,k}(a)$ exist in one of the stated
parameter ranges. By Lemma~\ref{lem:energy}, their positive
central amplitudes lie in $(\beta_a,1)$, and $|v_j(r)|\le v_j(0)<1$, $j=1,2$. 
Moreover, Proposition~\ref{prop:thresholds} yields
\[
\left(\frac{G_a}{g_a}\right)'(s)
\ge\frac{n-2}{2},\qquad s\in(\beta_a,1),\quad n=2,3,4.
\]
Since $2\le n\le4$, we have $(n-2)/2\le1$.
Thus Lemma~\ref{lem:extension}, with
$C_0=(n-2)/2$, provides a common odd $C^1$
extension $\widetilde g$ for which both states remain
solutions. {This extension satisfies the structural assumptions of \cite[Theorem~1.1]{CGHY}, as well as \eqref{eq:CGH-all} with $f=\widetilde g$, $F=\widetilde G$ and $\beta=\beta_a$ on $(\beta_a,\infty)$.} 
Since $v_1$ and $v_2$ have the same
number $k$ of zeros and positive central values,
the cited uniqueness theorem implies $v_1\equiv v_2$,
a contradiction.
Finally, Proposition~\ref{prop:thresholds} establishes
the sharpness of $a_3$ and $a_4$ for the corresponding
comparison inequalities. This completes the proof.
\end{proof}

\begin{proof}[Proof of Theorem~\ref{thm:one}]
We first determine the exact parameter ranges for
\eqref{eq:one-criterion}.
By Proposition~\ref{prop:tangent}, the tangent condition
holds precisely when $a\ge a_T$.
Moreover, Proposition~\ref{prop:thresholds} shows that
$m$ is continuous and strictly increasing on $(0,a_3]$,
with $m(a_3)=1/2$, and $m(a)\ge1/2$ for $a\ge a_3$.

For $3\le n\le23$, by \eqref{eq:tangent-dimension}, we have
\[
\frac{n-2}{2n}\le\frac{21}{46}<m(a_T).
\]
Hence both conditions in \eqref{eq:one-criterion}
hold exactly when $a\ge a_T$.
For $n\ge24$,
\[
m(a_T)<\frac{11}{24}
\le\frac{n-2}{2n}<\frac12=m(a_3).
\]
By continuity and strict monotonicity, there exists
a unique $\gamma_n\in(a_T,a_3)$ such that $m(\gamma_n)=(n-2)/(2n)$,
and \eqref{eq:one-criterion} holds exactly when
$a\ge\gamma_n$.
Since $(n-2)/(2n)$ increases strictly to $1/2$,
the sequence $(\gamma_n)$ is strictly increasing
and converges to $a_3$.

It remains to prove uniqueness.
Existence follows from Proposition~\ref{prop:existence}.
Suppose that two distinct states
$v_1,v_2\in\mathcal B_{n,1}(a)$ exist in one
of the stated ranges.
Proposition~\ref{prop:tangent} and
\eqref{eq:one-criterion} yield
\[
g_a(s)\ge(s-\beta_a)g_a'(s),\qquad
\left(\frac{G_a}{g_a}\right)'(s)
\ge\frac{n-2}{2n},
\quad s\in(\beta_a,1).
\]
Since $(n-2)/(2n)<1$, Lemma~\ref{lem:extension}
provides a common odd $C^1$ extension
$\widetilde g$ {satisfying the structural hypotheses of \cite[Theorem~1.2]{CGHY} and both inequalities in \eqref{eq:CGH-one} (with $f=\widetilde g$ and $F=\widetilde G$) on $(\beta_a,\infty)$.}
Both states remain solutions of the extended equation,
have exactly one zero, and are positive at the origin.
Thus \cite[Theorem~1.2]{CGHY} implies
$v_1\equiv v_2$, a contradiction.
Therefore $\mathcal B_{n,1}(a)$ consists of exactly
one state in the stated ranges, which are sharp
for the two comparison conditions
\eqref{eq:one-criterion}.
\end{proof}

\begin{proof}[Proof of Corollary~\ref{cor:small}]
Under the rescaling
\[
v(x)=\frac{a}{1+a}w(\sqrt a\,x),
\qquad
\varepsilon=\frac{a}{(1+a)^2},
\]
the equation $\Delta v+g_a(v)=0$ becomes
\[
-\Delta w+w=|w|w-\varepsilon|w|^2w
\quad\text{in }\mathbb R^n.
\]
This transformation preserves radiality, the number of
zeros, and positivity at the origin. 

{For $n=3,4,5$, we have $1<2<(n+2)/(n-2)$, as required in \cite[Theorem~1.1]{GJL}.}
Hence \cite[Theorem~1.1]{GJL}, with $p=2$ and $q=3$,
gives uniqueness of radial $k$-node bound states for
all sufficiently small $\varepsilon>0$.
Since
\[
\varepsilon'(a)=\frac{1-a}{(1+a)^3}>0,
\qquad \varepsilon(a)\to0\text{ as }a\to0^+,
\]
uniqueness follows for $0<a<a_{n,k}$ with some
$a_{n,k}>0$. Existence is guaranteed by
Proposition~\ref{prop:existence}.
\end{proof}

\vspace{.3cm}

\noindent\textbf{Conflict of interest.}
The authors declare that they have no conflict of interest.

\medskip
\noindent\textbf{Data availability.}
Data sharing is not applicable to this article as no data were created or analyzed in this study.

\end{document}